\documentclass[11pt]{amsart}
\usepackage[english]{babel}
\usepackage{amsmath, amsthm, amsfonts, amssymb}
\usepackage{latexsym}
\usepackage{bbm}
\usepackage{mathrsfs}
\usepackage{tikz}
\usepackage{graphicx}

\newtheorem{theorem}{Theorem}[section]
\newtheorem{proposition}[theorem]{Proposition}
\newtheorem{lemma}[theorem]{Lemma}
\newtheorem{corollary}[theorem]{Corollary}
\theoremstyle{definition}
\newtheorem{definition}[theorem]{Definition}

\theoremstyle{remark}

\numberwithin{equation}{section}
\numberwithin{theorem}{section}

\def\N{\mathbb{N}}

\def\R{\mathbb{R}}
\def\C{\mathbb{C}}

\def\a{\alpha}

\def\g{\gamma}
\def\s{\sigma}

\def\e{\varepsilon}

\def\lan{\langle}
\def\ran{\rangle}

\def\ti{\times}
\def\diag{\mathrm{diag}}
\def\proj{{\mathbb{P}}}
\def\Sph{\mathbb{S}}
\def\im{\mathrm{im}}
\def\Grass{\mathbb{G}}

\def\ug{\mathbb{U}}
\def\z{\zeta} 
\newcommand{\ang}{\sphericalangle}

\newcommand{\kercond}{\mathrm{kercond}} 
\newcommand{\imcond}{\mathrm{imcond}}
 
\newcommand{\cM}{\mathcal{M}} 
\newcommand{\cN}{\mathcal{N}}

\newcommand{\Hy}{\mathcal{H}}   
\newcommand{\Hyl}{\mathcal{H}_{\mathrm{lin}}}    
  
\newcommand{\vol}{\mathrm{vol}}    
\newcommand{\Prob}{\mathrm{Prob}}
\newcommand{\E}{\mathbb{E}}  
    
\newcommand{\Reg}{\mathrm{Reg}}  
\newcommand{\Sing}{\mathrm{Sing}}

\newcommand{\spann}{\mathrm{span}}

\newcommand{\Irrel}{\mathrm{Irrel}}
\newcommand{\suf}{\mathrm{sf}}

\newcommand{\rdeg}{\mathrm{rdeg}}
\newcommand{\codim}{\mathrm{codim}}

\title[Condition of intersecting a projective variety with a linear subspace]{Condition of intersecting a projective variety\\ 
with a varying linear subspace} 

\author{Peter B\"urgisser}

\thanks{
Institute of Mathematics, Technische Universit\"at Berlin, 
pbuerg@math.tu-berlin.de.
Partially supported by DFG grant BU 1371/2-2.}

\date{\today}



\begin{document}
\maketitle

\begin{abstract}
The numerical condition of the problem of intersecting 
a fixed $m$-dimensional irreducible complex projective variety $Z\subseteq\proj^n$
with a varying linear subspace $L\subseteq\proj^n$ of complementary 
dimension $s=n-m$ is studied. 
We define the intersection condition number~$\kappa_Z(L,z)$ 
at a smooth intersection point $z\in Z\cap L$ as the norm of 
the derivative of the locally defined solution map 
$\Grass(s,\proj^n)\to\proj^n,\, L\mapsto z$.
We show that $\kappa_Z(L,z) = 1/\sin\a$, 
where $\a$~is the minimum angle between the tangent spaces $T_zZ$ and $T_zL$.
From this, we derive a {\em condition number theorem} that 
expresses $1/\kappa_Z(L,z)$ as 
the distance of $L$ 
to the local Schubert variety, which consists of the linear subspaces 
having an ill-posed intersection with~$Z$ at~$z$. 
A probabilistic analysis of the maximum condition number 
$\kappa_Z(L) := \max \kappa_Z(L,z_i)$, taken over all intersection points 
$z_i\in Z\cap L$, leads to the study of the volume of tubes around 
the Hurwitz hypersurface $\Sigma(Z)$. 
As a first step towards this, 
we express the volume of $\Sigma(Z)$ in terms of its degree.
\end{abstract}

\section{Introduction}

Let $Z\subseteq\proj^n$ be a fixed $m$-dimensional irreducible complex projective variety with $0<m<n$, and 
$L\subseteq\proj^n$ be a linear subspace of complementary dimension $s=n-m$.
B\'ezout's theorem tells us that if $L$ is in sufficiently general position, then the intersection $Z\cap L$ 
consists exactly of degree of~$Z$ many points. 
In fact, in numerical algebraic geometry~\cite[\S13.3]{sommese-wampler:05}, 
the variety~$Z$ is described by a {\em witness point set}, which is nothing but 
$Z\cap L$ for a ``generic'' subspace $L$. 
Along with the witness point set, 
one needs a routine to keep track of $Z\cap L$ 
when $L$ moves (e.g., by describing a loop in the Grassmann manifold). 
This routine is usually implemented by a Newton homotopy continuation.

Numerical computations are affected by errors (e.g., due to round-off), so it is important to understand 
to what extent the witness point set $Z\cap L$ changes when $L$ is perturbed. 
In this paper, we achieve this by 
introducing and analyzing the notion of an {\em intersection condition number}.
In doing so, we follow the general geometric framework as introduced by 
Rice~\cite{Rice} and Shub and Smale~\cite{Bez4}; see also \cite[Chap.~14]{condition}. 
We see our work as a step towards a better theoretical understanding of the algorithms 
in numerical algebraic geometry. 

\subsection{Kernel intersection condition number}

Suppose that $L\subseteq\proj^n$ corresponds to the 
kernel $\widehat{L}$ of a full rank matrix $A\in\C^{m\times (n+1)}$,  
so that $\dim\hat{L}=s+1$. 
If $L$ has a transversal intersection with $Z$ at the smooth point $z$ of $Z$, 
then the implicit function theorem shows that there is a holomorphic map~$G$ 
sending matrices~$\tilde{A}$ in an open neighborhood~$U$ of $A$ to points $G(\tilde{A})$ in 
an open neighborhood~$V$ of $z$ in $Z$ such that 
$\proj(\ker\tilde{A}) \cap V = \{ G(\tilde{A})\}$ for all $\tilde{A}\in U$. 
We shall call~$G$ the {\em solution map} at $(A,z)$. 

The space $\proj^n$ is a compact complex manifold with a unitary invariant hermitian 
metric on its tangent spaces. We take 
$T_{z} \proj^n :=\{ \dot{z}\in \C^{n+1}\mid \lan \dot{z}, z\ran = 0\}$
as a model for the tangent space at the point represented by $z\in\C^{n+1}\setminus\{0\}$ 
and use the hermitian inner product 
$\lan \dot{u}, \dot{v}\ran_z := \|z\|^{-2}\lan \dot{u},\dot{v}\ran$, where
$\lan \dot{u},\dot{v}\ran :=\sum_j \dot{u}_j\bar{\dot{v}}_j$ for $\dot{u},\dot{v}\in\C^{n+1}$ 
denotes the standard hermitian inner product on $\C^{n+1}$
(Fubini-Study metric; see \cite[\S14.2]{condition}).

Consider now the derivative 
$D_A G\colon T_A \C^{m\times (n+1)} \to T_z \proj^n$ 
of the solution map $G$ at $A$ and its operator norm 
$$
 \| D_A G \| := \sup_{\|\dot{A}\|_F =1} \| D_AG (\dot{A}) \| ,
$$
defined with respect to the Frobenius norm 
$\|\dot{A}\|_F := (\sum_{ij} |\dot{a}_{ij}|^2 )^{1/2}$
on $\C^{m\times (n+1)}$.

\begin{definition}\label{def:condA}
Let $Z\subseteq\proj^n(\C)$ be a fixed $m$-dimensional irreducible projective variety. 
Let $A\in\C^{m\times (n+1)}$ be of full rank and $z\in Z$ such that $Az=0$. 
The {\em kernel intersection condition number} of~$A$ at $z$
(with respect to the variety $Z$) is defined as 
$\kercond_Z(A,z) := \|A\|  \cdot \|D_A G\|$ 
if $z$ is a smooth point of~$Z$
and if $\proj(\ker A)$ has a transversal intersection with $Z$
at~$z$. 
(Here $\|A\|$ denotes the spectral norm of~$A$.) 
We set $\kercond_Z(A,z) := \infty$ in all the other cases. 
\end{definition} 

Note that by this definition, the condition number is scale invariant: 
$\kercond_Z(t A,z) = \kercond_Z(A,z)$ for $t\in\C^\times$. 

\subsection{Intrinsic intersection condition number}
\label{se:intrinsic}

In order to understand the kernel intersection condition number, 
it is useful to work with a more intrinsic notion of condition, 
following ideas in~\cite{amel-pbuerg:11a}. 
The complex Grassmann manifold
$\Grass:=\Grass(s,\proj^n)$ is defined as the set of $s$-dimensional projective linear 
subspaces $L$ of $\proj^n$. 
It is a compact complex manifold with a unitary invariant hermitian
metric on its tangent spaces, 
see Section~\ref{se:Grass-conv}. 
Clearly, we may identify $\Grass$ with the set $\Grass(s+1,\C^{n+1})$ 
of linear subspaces $\widehat{L}$ of $\C^{n+1}$ having the dimension $s+1$. 

Assume that $z\in \Reg(Z)$ is a smooth point of the 
$m$-dimensional projective variety $Z\subseteq\proj^n$.  
Again, we put $s=n-m$. 
Moreover, assume that $L\in\Grass$ intersects~$Z$  
transversally at~$z$, 
which means that $T_zZ + T_zL = T_z\proj^n$. 
The implicit function theorem implies that there is a holomorphic map
\begin{equation}\label{eq:gamma}
 \g\colon\Grass\supseteq U \to V\subseteq Z
\end{equation}
sending spaces~$\tilde{L}$ in an open neighborhood~$U$ of $L$ 
to points~$\g(\tilde{L})$ in an open neighborhood~$V$ of $z$ in $Z$ 
such that 
$\tilde{L}\cap V = \{ \g(\tilde{L})\}$ for all $\tilde{L}\in U$. 
We shall call $\g$ the {\em solution map} at $(L,z)$. 
Consider now the derivative 
$D_L \g\colon T_L \Grass \to T_z \proj^n$ 
of the map~$\g$ at $L$ and its spectral norm 
$$
 \| D_L \g\| := \sup_{\|\dot{L}\| =1} \| D_L \g (\dot{L}) \| ,
$$
with respect to the hermitian norm on $T_L\Grass$ 
and the Fubini-Study norm on $T_z \proj^n$. 

\begin{definition}\label{def:condG}
Let $Z\subseteq\proj^n(\C)$ be a fixed $m$-dimensional irreducible
projective variety and $s:=n-m$.
Let $L\in \Grass(s,\proj^n)$ and $z\in Z$ be such that $z\in L$. 
The {\em intersection condition number} of $L$ at $z$
(with respect to the variety $Z$) is defined as 
$\kappa_Z(L,z) := \|D_L \g\|$ 
if $z$ is a smooth point of~$Z$ and if $L$ has a transversal intersection with $Z$ at~$z$. 
We set $\kappa_Z(L,z) := \infty$ in all the other cases. 
\end{definition} 

The following result  is inspired by~\cite{amel-pbuerg:11a} and reveals the connection 
between the kernel intersection condition number $\kercond_Z(A,z)$, when $L$ is given as the kernel of a matrix $A$, 
and the intrinsic condition number~$\kappa_Z(L,z)$. 
We recall that the condition number of the (full rank) matrix $A\in\C^{m\times (n+1)}$ 
is defined as $\kappa(A) := \|A\|\cdot \|A^\dagger\|$, where 
$A^\dagger$ denotes the Moore-Penrose inverse of~$A$ 
and we use spectral norms. 

\begin{theorem}\label{pro:kckA}
We have 
$\kappa_Z(L,z) \le \kercond_Z(A,z) \le \kappa(A)\cdot\kappa_Z(L,z)$
if $L\in\Grass(s,\proj^n)$ is given as the kernel of the full rank
matrix $A\in\C^{m\times (n+1)}$ and $z\in L$.
\end{theorem}

This result, whose proof is provided in Section~\ref{se:kerimcondkappa}, 
shows that $\kercond_Z(A,z)$ can be thought of as being composed of 
the intrinsic condition $\kappa_Z(L,z)$ and of the matrix condition $\kappa(A)$,
where the latter only depends on the way the subspace~$L$ is represented.
In particular, we see that $\kappa_Z(L,z) = \kercond_Z(A,z)$ if $A$ has the minimal condition $\kappa(A)=1$.

We can similarly define an intersection condition number $\imcond_Z(B,z)$, 
when representing the space $\hat{L}$ as the image of $B$, where $B\in\C^{(s+1) \ti (n+1)}$ 
is a full rank matrix. 
For this, a result analogous to Theorem~\ref{pro:kckA} holds, see~\eqref{eq:imcond}.

\subsection{Geometric characterization}
\label{se:geo-char}

We shall provide two geometric characterizations of the intersection condition number $\kappa_Z(L,z)$. 
The first one is local and characterizes $\kappa_Z(L,z)$ as the
inverse of the sine of the minimum angle 
between $T_zL$ and the tangent space $T_{z}Z$. 
This result reminds of the Grassmann condition number of a 
convex cone at a linear subspace, cf.~\cite{amel-pbuerg:11a}. 

The second characterization is global and can be seen as a {\em condition number theorem} 
in the spirit of Eckhart-Young~\cite{eckyoung:36}, Demmel~\cite{Demmel87}, Shub and Smale~\cite{bez1}; 
see \cite{condition} for this and further results of this type. 
We characterize $\kappa_Z(L,z)$ as the inverse of 
the distance of $L$ in the Grassmann manifold to the set of~$\tilde{L}$ 
that intersect $Z$ nontransversally at $z$ (the latter may be considered as the``ill-posed''  
instances of the computational problem). 
Let us remark that the Eckhart-Young Theorem, which is the most familiar variant of such a result, 
characterizes the usual matrix condition number in this way. 

Let $V$ be a finite dimensional hermitian vector space and 
$V_1,V_2\subseteq V$ be linear subspaces of dimensions $m_1,m_2$, respectively.
It is a well-known fact, essentially due to Jordan~\cite{jordan:1875}, 
that the relative position of $V_1$ and $V_2$ 
is determined by the {\em principal angles} $\theta_1,\ldots,\theta_r$ between $V_1$ and~$V_2$, 
where $r:=\min\{m_1,m_2\}$
(see Section~\ref{se:principal-angles} for more details). 
The (principal) angle between two complex lines $\C v_1$ and $\C v_2$ is given by 
$\ang (v_1,v_2) := \arccos\frac{|\lan v_1,v_2\ran |}{\|v_1\|\cdot\| v_2\|}$. 
One can show that 
\begin{equation}\label{eq:def-angle}
 \ang_{\min}(V_1,V_2) := 
  \min\big\{ \ang(v_1,v_2) \mid v_1\in V_1\setminus\{0\}, v_2\in V_2\setminus\{0\}\big\} 
 = \min_j\theta_j ,
\end{equation}
and we call $\ang_{\min}(V_1,V_2)$ the {\em minimum angle} between $V_1$ and~$V_2$. 

Assume now $m_1=m_2=:m$.
We will consider two distance measures on the 
Grassmann manifold~$\Grass(m,V)$ of $m$-dimensional 
linear subspaces of $V$. 
The {\em projection distance} between $V_1$ and~$V_2$ 
is defined as 
$d_p(V_1,V_2) := \| \Pi_{V_1} - \Pi_{V_2}\|$, where 
$\Pi_{V_i}\colon V\to V_i$ denotes the orthogonal projection onto $V_i$; 
cf.~\cite[section~2.6]{golloan:83}.
It is known that 
$d_p(V_1,V_2) = \sin\max\{\theta_1,\ldots,\theta_r\}$;
see \cite[Section~12.4.3]{golloan:83} or \cite[Section~5.3]{Stewart}.

Since $\Grass(m,V)$ is a compact Riemannian manifold, we can define 
the {\em geodesic distance} $d_g(V_1,V_2)$ as 
the minimum length of curves in $\Grass(m,V)$ connecting $V_1$ and $V_2$. 
It is known that $d_g(V_1,V_2) = \sqrt{\theta_1^2+\cdots+\theta_r^2}$; 
see~\cite{Wo:67}. Moreover, $0\le d_g(V_1,V_2)\le\pi/2$. 

We return now to our setting of an irreducible projective variety 
$Z\subseteq\proj^n$. Let $z\in Z$ be a smooth point. 
The following result shows that 
$1/\kappa_Z(L,z)$ equals the sine of the minimum angle between $T_zZ$ and~$T_zL$ 
(interpreted as subspaces of $T_z\proj^n$).

\begin{theorem}\label{th:kappa-alpha}
Let $z\in Z\cap L$ be a smooth point of $Z$ and 
$\a$ be the minimum angle between $T_zZ$ and $T_zL$. 
Then we have $\kappa_Z(L,z) = 1/\sin\a$.
\end{theorem}

This result is quite intuitive: a small minimum angle means that there is a 
``glancing intersection'' of $L$ and $Z$, which results in a large intersection 
condition number (which is numerically undesirable). 
Note that we have a transversal intersection of $L$ and $Z$ 
at a point $z\in L\cap Z$ iff this minimum angle is positive. 
We provide the proof in Section~\ref{se:ang_char}. 

As an immediate consequence of Theorem~\ref{th:kappa-alpha} 
we conclude that $\kappa_Z(L,z) \ge 1$. 

\subsection{Inverse distance to ill-posedness}

In the following, we again write 
$\Grass := \Grass(s,\proj^n)$. 
For $z\in\proj^n$ we consider the set 
$\Grass_z := \{L\in\Grass \mid z \in L\}$
of $s$-dimensional projective linear subspaces passing through~$z$,
which can be identified with the Grassmann manifold of 
$(s-1)$-dimensional linear subspaces of $\proj(T_z\proj^n)$, 
since $T_z\proj^n \simeq \C^{n+1}/\C z$. 
Again, let $Z\subseteq\proj^n$ be a fixed $m$-dimensional irreducible projective variety 
and $s+m=n$. 

\begin{definition}\label{def:local-schubert}
Let $z$ be a smooth point of $Z$. 
The {\em local Schubert variety} $\Sigma_z(Z)$ of $Z$ at~$z$
consists of the $L\in\Grass_z$ having a nontransversal intersection 
with $Z$ at $z$, that is, $T_zZ \cap T_z L \ne 0$. 
\end{definition}

So $\Sigma_z(Z)$ consists exactly of the $L\in\Grass_z$ satisfying 
$\kappa_Z(L,z)=\infty$. The set 
$\Sigma_z(Z)$ is a closed subset of $\Grass_z$ 
since $T_zZ \cap T_z L \ne 0$ is equivalent to $\dim (T_zZ+T_zL) < n$,
which can be expressed by the vanishing of minors.

The geodesic distance~$d_g$ and the projection distance~$d_p$ 
both define a metric on the subspace~$\Grass_z$ of $\Grass$.  
(In fact, one can show that $d_g(L_1,L_2)$ equals the geodesic distance between 
$L_1$ and $L_2$ in the Riemannian manifold $\Grass_z$.) 
For $L\in\Grass_z$ we write
$d_g(L,\Sigma_z(Z)) := \min \{ d_g(L,L') \mid L' \in \Sigma_z(Z)\}$
and we define $d_p(L,\Sigma_z(Z))$ analogously. 

We can now state the announced condition number theorem. 
The proof relies on Theorem~\ref{th:kappa-alpha}, 
uses an idea from~\cite{amel-pbuerg:11a}, 
and is provided in Section~\ref{se:pf_CNT}.

\begin{theorem}[Condition Number Theorem]\label{th:CNT}
Let $z$ be a smooth point of $Z$ and $L\in\Grass(s,\proj^n)$ 
be such that $z\in L$. Then we have 
$d_p(L,\Sigma_z(Z)) = \sin d_g(L,\Sigma_z(Z))$ and 
$$
  \kappa_Z(L,z) = \frac{1}{\sin d_g(L,\Sigma_z(Z))} .
$$
\end{theorem}

Definition~\ref{def:condG} introduced the intersection condition number 
$\kappa_Z(L,z)$ at a point $z\in Z\cap L$. We now make a global definition. 
In order to avoid the discussion of mathematical subtleties not relevant for our purposes, 
we define $\Irrel(Z)$ as the set of all $L\in\Grass$ with the property 
that $\Sing(Z) \cap L \ne \emptyset$ or $Z\cap L$ has positive dimension. 
We will consider all $L\in\Irrel(Z)$ as ill-posed. 
It is easy to see that $\Irrel(Z)$ is contained in an algebraic hypersurface of $\Grass$
and thus has the measure zero. 

\begin{definition}\label{def:kappa-global}
The {\em intersection condition number} of $L\in\Grass\setminus\Irrel(Z)$ 
(with respect to the variety~$Z$) is defined as 
$$
 \kappa_Z(L) \ := \max_{z\in Z\cap L} \kappa_Z(L,z) .
$$
Moreover, we define $\kappa_Z(L) = \infty$ if $L\in\Irrel(Z)$. 
We say that $L$ is {\em ill-posed} iff $\kappa_Z(L) = \infty$.  
\end{definition}

In order to understand the set of ill-posed $L$ in a geometric way, 
let us make the following definition 
(with follows the naming in~\cite{sturmfels:14}). 

\begin{definition}\label{def:hurwitz}
The {\em Hurwitz variety} $\Sigma(Z)$ of $Z$ is defined as the Zariski closure of 
the union of the local Schubert varieties $\Sigma_z(Z)$, 
taken over all regular points~$z$ of $Z$. 
\end{definition}

In this definition 
we may as well take the closure with respect to the Euclidean topology, 
since the union of the local Schubert varieties 
is a constructible set, see \cite[\S 2C]{mumford}.
We note that $\Sigma(Z) = \Irrel(Z)$ if $\deg Z = 1$, and 
$\Sigma(Z)$ is not a hypersurface in this case 
(compare Theorem~\ref{pro:sturmfels} below).

We provide the proof of the following auxiliary result at the beginning  
of Section~\ref{se:pf_CNT}. 

\begin{lemma}\label{le:Sigma}
If $L\in\Sigma(Z)\setminus\Irrel(Z)$, then there exists a smooth $z\in Z\cap L$ 
such that $L\in\Sigma_z(Z)$. In particular, if $Z$ is smooth, then $\Sigma(Z)$
equals the union of the $\Sigma_z(Z)$ over $z\in Z$.
\end{lemma}

This lemma implies that $\kappa_Z(L)=\infty$ for all $L\in\Sigma(Z)$.
Hence we obtain from Definition~\ref{def:kappa-global}: 
$$
 \{ L\in \Grass \mid \kappa_Z(L) = \infty \} = \Sigma(Z) \cup \Irrel(Z) .
$$

We define the geodesic distance of $L$ to $\Sigma(Z)$ as 
$d_g(L,\Sigma(Z)) := \min \{ d_g(L,L') \mid L' \in \Sigma(Z)\}$. 
Moreover, for $\e\ge 0$, we define the $\e$-neighborhood (or $\e$-tube) around $\Sigma(Z)$ by 
\begin{equation}\label{eq:def-tube}
 T(\Sigma(Z),\e) \ :=\ \big\{ L \in \Grass \mid d_g(L,\Sigma) \le \arcsin\e \big\} .
\end{equation}

\begin{corollary}\label{cor:kappa-tube}
We have 
$\{ L\in\Grass \mid \kappa_Z(L) \ge \e^{-1} \} \subseteq T(\Sigma(Z), \e) \cup \Irrel(Z)$. 
\end{corollary}

\begin{proof}
Let $L\not\in\Irrel(Z)$ with $\kappa_Z(L) \ge \e^{-1}$. 
By Definition~\ref{def:kappa-global}, there exists 
$z\in Z\cap L$ such that $\kappa_Z(L) = \kappa_Z(L,z)$. 
The point $z$ is smooth since $L\not\in\Irrel(Z)$. 
Therefore, 
$\sin d_g(L,\Sigma(Z)) \le \sin d_g(L,\Sigma_z(Z)) = \kappa_Z(L,z)^{-1} \le \e$, 
where we used Theorem~\ref{th:CNT} for the last equality. 
\end{proof}

\subsection{Volume of hypersurfaces in Grassmannians}
\label{se:vtub}

Again let $\Grass=\Grass(s,\proj^n)$. 
Recall the Pl\"ucker embeddding (e.g., see~\cite[\S3.1]{GKZ}): 
\begin{equation}\label{eq:plucker}
\iota\colon\Grass \hookrightarrow \proj(\Lambda^{s+1}\C^{n+1}),\
 W \mapsto \proj(\Lambda^{s+1}W) .
\end{equation}
Let $\Hy$ be an irreducible hypersurface in $\Grass$. 
It is known that the vanishing ideal of~$\iota(\Hy)$ 
in the homogeneous coordinate ring of $\iota(\Grass)$ 
is principal (cf.~\cite[Chap.~3, Prop.~2.1]{GKZ}). 
We call the degree of its irreducible generator the 
{\em relative degree} of $\Hy$ and denote it by $\rdeg\Hy$.
Thus $\iota(\Hy)$ is obtained by intersecting 
$\iota(\Grass)$ with an irreducible hypersurface 
of degree $\rdeg\Hy$ 
in the projective space $\proj(\Lambda^{s+1}\C^{n+1})$. 
Therefore, by B\'ezout's theorem, we have 
\begin{equation}\label{eq:rdeg}
\deg\iota(\Hy) = \rdeg\Hy\cdot \deg\iota(\Grass).
\end{equation}

The regular locus of a hypersurface $\Hy$ of $\Grass$ is a smooth submanifold of $\Grass$ 
and hence it has a well defined volume, defined via the restriction of the 
Riemannian metric on $\Grass$, that we denote by $\vol(\Hy)$. 

Consider now the Chow variety 
\begin{equation}\label{eq:H-lin}
 \Hyl := \{ L\in\Grass \mid \proj^{m-1} \cap L \ne \emptyset \} 
\end{equation}
of a fixed linear subspace $\proj^{m-1}\subseteq\proj^n$ of dimension~$m-1$. 
It is easy to see that $\Hyl$ is an irreducible hypersurface of $\Grass$ of degree one. 
In a sense, these are the ``simplest'' hypersurfaces of  $\Grass$. 
They are special Schubert varieties. 
(It is known that $\Hyl$ is singular except when $m=1$, cf. \cite[\S3.4.1]{manivel:01})

We present the short proof of the following result in Section~\ref{se:vol-HS}
and remark that a previous version of our paper contained a 
direct, but technically involved proof of this result. 

\begin{theorem}\label{th:HlinC}
\begin{enumerate}
\item An algebraic hypersurface $\Hy$ of $\Grass$ satisfies 
$\vol(\Hy) = \rdeg\Hy \cdot \vol(\Hyl) $. 

\item The volume of the special Schubert variety $\Hyl$ in $\Grass$ satisfies
\begin{equation*}
 \frac{\vol(\Hyl)}{\vol(\Grass)}
  = \frac{1}{\pi}\cdot \dim\Grass
  = \frac{s+1}{\sqrt{\pi}}\cdot\frac{m}{\sqrt{\pi}} 
  = \pi \cdot \frac{\vol(\proj^s)}{\vol(\proj^{s+1})} \cdot \frac{\vol(\proj^{m-1})}{\vol(\proj^{m})} .
\end{equation*}
\end{enumerate}
\end{theorem}

The following result is due to Sturmfels~\cite{sturmfels:14}. 

\begin{theorem}\label{pro:sturmfels}
The Hurwitz variety $\Sigma(Z)$ is an irreducible hypersurface in $\Grass$ if $\deg Z\ge 2$. 
Moreover, if $\Sing(Z)$ has codimension in~$Z$ at least two, then 
the relative degree of $\Sigma(Z)$ satisfies 
$\rdeg\Sigma(Z) = 2\deg Z + 2g(Z) -2$, 
where $g(Z)$ is the sectional genus of $Z$, i.e., 
the arithmetic genus of the intersection of $Z$ with a generic linear subspace 
of codimension $s-1$.
\end{theorem}

Sturmfels' work~\cite{sturmfels:14} focused on the irreducible generator of $\Sigma(Z)$, 
for which he coined the name ``Hurwitz form'', due to the apparent similarity 
of the formula for $\deg\Sigma(Z)$ with the Riemann-Hurwitz formula.

The following corollary is an immediate consequence of
Theorem~\ref{pro:sturmfels} and Theorem~\ref{th:HlinC}.

\begin{corollary}\label{cor:vol_Hurwitz}
Let $Z\subseteq\proj^n$ be an irreducible projective variety 
such that $\deg(Z)\ge 2$ and the codimension of $\Sing(Z)$ in $Z$ is at least two. 
Let $g(Z)$ denote the sectional genus of $Z$. Then we have
$$
 \frac{\vol(\Sigma(Z))}{\vol(\Grass)} \ =\ 
 \frac{2}{\pi} \, (\deg(Z) + g(Z) -1) \cdot \dim Z \cdot (\codim Z +1 ) . 
$$
\end{corollary}

\subsection{Towards a probabilistic analysis} 
\label{se:vol-tubes}

Suppose that $L\in\Grass$ is randomly chosen with respect to the uniform distribution 
on $\Grass$. We would like to show that it is unlikely that $\kappa_Z(L)$ is large. 
By Corollary~\ref{cor:kappa-tube}, this reduces to upper bounding 
the volume of the tubes around the Hurwitz variety $\Sigma(Z)$. 
(Since $\Irrel(Z)$ has measure zero, it is clearly irrelevant for this task, hence 
its naming.) 
More specifically, the goal is to establish, for $0<\e \le 1$ and $L\in\Grass$ chosen uniformly at random,
an upper bound on the tail probability of the form 
$$
 \Prob_{L\in\Grass} \{ \kappa_Z(L) \ge \e^{-1} \} \ \le\ 
 \frac{\vol(T(\Sigma(Z),\e))}{\vol(\Grass)} \ \le \ K \e^2 ,
$$
where $K$ is polynomially bounded in the dimension~$n$ and 
the degree of $\Sigma(Z)$. 
Such a bound implies 
$\E_{L\in\Grass}(\kappa_Z) = \int_1^\infty \Prob_{L\in\Grass}\{ \kappa_Z(L) \ge t \}\, dt 
    \ \le\  K \int_1^\infty \frac{dt}{t^2} = K$ 
and hence  
$\E(\log\kappa_Z) \le \log \E(\kappa_Z) \le \log K$. 

In~\cite{BCL:06}, this program was carried out for the volume of the tube  
around a subvariety of complex projective space. 
It would be interesting to extend this from projective spaces to Grassmannians. 
Since for a hypersurface $\Hy$ of $\Grass$ we have 
$\vol(T(\Hy,\epsilon)) = \vol(\Hy) \pi \epsilon^2 + o(\epsilon^2)$ 
in first order approximation, Corollary~\ref{cor:vol_Hurwitz} 
provides a first step towards this task. 

We mention that for a smooth irreducible hypersurface $\Hy$ of $\Grass$, 
Gray~\cite[Thm. 1.1(i)]{gray:84} proved the upper bound 
$\vol(T(\Hy,\sin\theta)) \le  \vol(\Hy)\cdot \pi \theta^2$
if $\theta$ is smaller than the geodesic distance from $\Hy$ 
to its nearest focal point.
Unfortunately, the Hurwitz hypersurfaces are not smooth in general, 
so that this result cannot be used for our purposes. 

The task of bounding the volume of tubes around hypersurfaces in complex 
Grassmannians will be studied in a follow-up paper.

\subsection{Acknowledgments.} 
I thank Sameer Agarwal, Hon Leung Lee, Rekha Thomas, and 
Bernd Sturmfels for suggesting these investigations, 
motivated by their work in algebraic vision~\cite{alts:14}. 
Special thanks go to Dennis Amelunxen for his detailed comments 
and suggestions for improvement of an earlier version of this work.
I also thank Paul Breiding, Kathl\'en Kohn and Pierre Lairez for 
helpful discussions and comments on the manuscript. 
I am grateful for the financial support and inspiring working atmosphere 
at the Simons Institute for the Theory of Computing, where the basis 
of this work was laid.  

\section{Preliminaries}
\label{se:preliminaria}

\subsection{Unitary groups}\label{se:ug}

The unitary group 
$\ug(n) :=\{Q\in\C^{n\times n} \mid Q^*Q=I_n\}$ 
is a compact smooth submanifold of $\C^{n\times n}$. 
Its tangent space at $I_n$ is given by the space of skew-hermitian matrices
$$
 T_{I_n}\ug(n) = \{A\in\C^{n\times n} \mid A^* + A = 0 \}.
$$
We define an inner product on $T_{I_n}\ug(n)$ 
by setting for $A=[a_{ij}], B=[b_{ij}]$: 
\begin{equation}\label{eq:ug-can-metric}
\langle A,B\rangle := \sum_{i} \Im(a_{ii}) \Im(b_{ii}) 
  + \frac12 \sum_{i\ne j} a_{ij} \bar{b}_{ij} 
  = \sum_{i} \Im(a_{ii}) \Im(b_{ii})  + \sum_{i< j} a_{ij} \bar{b}_{ij} ,
\end{equation}
where $\Im(z)$ denotes the imaginary part of $z\in\C$. 
An orthonormal basis of $T_{I_n}\ug(n)$ is given by 
$$
 \big\{ E_{ij} -E_{ji} \mid i<j\big\} \cup
 \big\{ \sqrt{-1}(E_{ij} + E_{ji}) \mid i< j \big\} \cup 
 \big\{ \sqrt{-1} E_{ii} \mid i \big\} ,
$$
where $E_{ij}$ denote the standard basis vectors of $\C^{n\times n}$.  
We extend this to a Riemannian metric on $\ug(n)$
by requiring that the left-multiplications are isometries 
and call the resulting metric the canonical one. 
It is important to realize that this metric is essentially different from the 
Riemannian metric on $\ug(n)$ that is induced by the Euclidean metric 
of $\C^{n\times n}$.  
(The reason is the contribution in~\eqref{eq:ug-can-metric}
from the imaginary elements on the diagonal; 
in the analogous situation of the orthogonal group, the Riemannian metrics 
differ by a constant factor only.) 
The reason to select the canonical metric is that for 
$v$ in the unit sphere $S(\C^n)$ of $\C^n$, the orbit map 
$\ug(n) \to S(\C^n),\, Q\mapsto Qv$ 
is a Riemannian submersion 
(which is easy to check).  
This implies $\vol(\ug(n)) = \vol(\ug(n-1))\, \vol(S(\C^n))$. 
We denote by $S^{m-1} := \{ x\in \R^{m} \mid \|x\| = 1 \}$
the unit sphere of $\R^{m}$.
Using $\vol(S(\C^n))= \vol(S^{2n-1}) = 2\pi^n/\Gamma(n)$, 
we obtain from this 
$\vol(\ug(n)) = 2^n \pi^{\frac{n(n+1)}{2}}/\suf(n-1)$,
where the {\em superfactorial} of $k\in\N$ is defined as 
$\suf(k) := \prod_{i=1}^k i!$ for $k\ge 1$ and $\suf(0) :=1$.  

\subsection{Complex Grassmann manifolds}
\label{se:Grass-conv}

For the following compare~\cite{eas:99}. 
Let $0 < m < n$. 
The {\em complex Grassmann manifold} $\Grass(m,\C^n)$ 
consists of the $m$-dimensional complex linear subspaces of $\C^n$. 
The unitary group $\ug(n)$ 
acts transitively on $\Grass(m,\C^n)$,
and we have a surjective and $\ug(n)$-equivariant mapping 
\begin{equation}\label{eq:pi-map}
 \pi\colon \ug(n) \to \Grass(m,\C^n),\ Q \mapsto [Q] := \spann\{q_1,\ldots,q_{m}\} ,
\end{equation}
where $q_i$ denote the columns of $Q$. 
We can thus interpret $\Grass(m,\C^n)$ as the quotient of $\ug(n)$ by the subgroup $\ug(m)\times\ug(n-m)$. 
For our puposes, this is the most effective way to operate with the Grassmann manifolds.

We identify the tangent space of $\Grass(m,\C^n)$ at $[I_n]$ with the orthogonal complement of 
the kernel of the derivative of~$\pi$ at $I_n$, which consists of the matrices
\begin{equation}\label{eq:TS-Gr}
\begin{bmatrix}
0 & -R^* \\
R & 0
\end{bmatrix}, \quad R \in \C^{(n-m)\times m} .
\end{equation}
We define a hermitian inner product on this space by 
\begin{equation}\label{eq:def-hip}
\Big\langle 
\begin{bmatrix}
0 & -R^* \\
R & 0
\end{bmatrix},
\begin{bmatrix}
0 & -S^* \\
S & 0
\end{bmatrix}
\Big\rangle 
 := \langle R, S \rangle := \mathrm{tr}(RS^*).
\end{equation}
By requiring that $\ug(n)$ acts in a unitarily invariant way, 
we get a hermitian metric on the tangent bundle of $\Grass(m,\C^n)$. 
It is easy to check that 
$\pi\colon\ug(n) \to \Grass(m,\C^n)$
is a Riemannian submersion. 
This implies
$\vol(\Grass(m,\C^n)) = \vol(\ug(n)/(\vol(\ug(m))\vol(\ug(n-m)))$.

\subsection{Principal angles between subspaces}
\label{se:principal-angles}

We are interested in characterizing the relative position of two subspaces 
and therefore study the orbits of the simultaneous action 
of $\ug(n)$ on $\Grass(m_1,\C^n)\times\Grass(m_2,\C^n)$, 
where $0\le m_1,m_2\le n$.
The principal angles, 
first introduced by  Jordan~\cite{jordan:1875} for subspaces of~$\R^n$, 
completely characterize the relative position of two subspaces. 
(For complex vector spaces, e.g., see~\cite{galhe:06}.)
The singular value decomposition and the related CS decomposition of matrices 
(cf.~\cite{paige-wei:94}) allow a short and elegant treatment.

Consider two complex linear subspaces $V_1$ and $V_2$ of $\C^n$ 
with the dimensions $m_1$ and~$m_2$, respectively,
and put $r:=\min\{m_1,m_2\}$.  
For $i=1,2$ choose an orthonormal basis of $V_i$ and form the matrix 
$U_i\in\C^{n\times m_i}$, whose columns consist of the basis vectors. 
Let $\s_1,\ldots,\s_r$ denote the singular values of the matrix product 
$U_1^*U_2\in\C^{m_1\times m_2}$. 
Note that 
$0\le\s_j\le \|U_1^*U_2\| \le 1$. 
We call $\theta_j := \arccos\s_j$, for $j=1,\ldots,r$,  
the {\em principal angles between the subspaces $V_1$ and $V_2$}. 
It is immediate to check from the unitary invariance of the singular 
values that this definition is independent of the choice of the orthonormal bases 
of $V_1$ and $V_2$. Moreover, the principal angles between $V_2$ and $V_1$ 
are the same as those between $V_1$ and $V_2$. It follows from the definition that 
if $Q\in\ug(n)$ is unitary, then the principal angles between $Q(V_1)$ and $Q(V_2)$ 
are the same as those between $V_1$ and $V_2$. 
It is known that the converse is true as well, so that the principal angles completely characterize 
the relative position of two subspaces.

\begin{lemma}\label{le:pr-ang-red}
Let $d:=\dim V_1\cap V_2$ and let $V'_i$ denote the orthogonal complement of $V_1\cap V_2$ in $V_i$. 
Then the list of principal angles between $V_1$ and $V_2$ is obtained from 
the list of principal angles between $V'_1$ and $V'_2$ by appending the angle~$0$ with multiplicity~$d$. 
\end{lemma}

\begin{proof}
For $i=1,2$, let $B_i$ be a matrix whose columns form an orthonormal basis of $V'_i$. 
Choose an orthonormal basis of $V_1\cap V_2$ and append it to the chosen bases of $V'_i$, 
obtaining the matrix~$U_i$. 
Then $U_1^*U_2$ has the block diagonal form $\diag(B_1^*B_2,I_d)$, which shows the assertion.
\end{proof}

\section{Proofs of characterizations of condition numbers}
\label{se:proofs_CN}

\subsection{Intrinsic versus kernel intersection condition numbers} 
\label{se:kerimcondkappa}

We provide here the proof of Theorem~\ref{pro:kckA}. 
Let $m+s=n$. Taking the orthogonal complement defines 
an isometry, cf.~\cite[eq.~(9)]{amel-pbuerg:11a}, 
\begin{equation}\label{eq:OC}
 \Grass' :=\Grass(m,\C^{n+1}) \to \Grass := \Grass(s+1,\C^{n+1}) ,\ W\mapsto W^\perp .
\end{equation}
Let $\cM\subseteq\C^{(n+1) \ti m}$ denote the open set of matrices of rank~$m$, 
endowed with the Frobenius norm.
We consider the smooth maps 
$$
\im\colon\cM\to\Grass',\, B\mapsto \im B ,\quad 
 \ker\colon \cM^T\to\Grass,\, A\mapsto \ker A .
$$
We study the operator norm $\|D_{B}\im\|$ of the derivative 
$D_B\im\colon T_B\cM \to T_{\im B} \Grass'$, defined 
with respect to the Frobenius norm on 
$T_B\cM= \C^{(n+1)\times m}$ and the norm on $T_{\im B} \Grass'$
defined in~\eqref{eq:TS-Gr}.  
The operator norm $\|D_A\ker\|$ is defined similarly. 

Since $(\im B)^\perp = \ker B^T$ for $B\in\cM$, 
we have the commutative diagram
\begin{equation}\label{eq:Mdiagram}
\begin{array}{ccl}
 \cM& \stackrel{\im}{\to }& \Grass' \\
   \downarrow &                         &  \downarrow\\
\cM^T & \stackrel{\ker}{\to} &   \Grass 
\end{array} ,
\end{equation}
where the vertical arrows are given by 
$B\mapsto A=B^T$ and $W\mapsto W^\perp$. 
This commutative diagram implies that the derivatives 
$D_{B}\im$ and $D_{A}\ker$ have the same operator norm. 

We shall identify $T_B\cM$  and $T_A\cM^T$ with $\cM$ and $\cM^T$, respectively. 
Recall that $B^\dagger$ denotes the Moore-Penrose inverse of $B$ 
and $\kappa(B) := \|B\|\| B^\dagger\|$. 
We have $\kappa(B) = \kappa(A)$. 

\begin{lemma}\label{le:diego}
Let $B\in\cM$ and $L'=\im(B)$. 
For all $\dot{B}\in \cM$  we have
$$
  \|D_B\im(\dot{B})\| \ \le\ \|B^\dagger\|\cdot \|\dot{B}\|_F,
$$ 
and there exists $\dot{B}\ne 0$ such that equality holds. 
Moreover, there is a linear subspace $\cN\subseteq \cM$ such that 
the restriction of $D_B \im$ to $\cN$ provides a bijection $\cN\to T_{L'}\Grass'$,  
and for all $\dot{B}\in \cN$: 
$$
\frac{1}{\|B\|} \,\|\dot{B}\|_F \ \le\ \|D_B \im (\dot{B})\| .
$$
Similarly, for $A=B^T$, $L=\ker(A)$, and $\dot{A}=\dot{B}^T$,  
we have 
$\|D_A \ker (\dot{A})\| \ \le\ \|A^\dagger\|\cdot \|\dot{A}\|_F$ 
and there is $\dot{A}\ne 0$ such that equality holds. 
Moreover, we have 
$\|\dot{A}\|_F / \|A\| \le  \|D_A \ker (\dot{A})\|$ 
for all $\dot{A} \in\cN^T$. 
\end{lemma}

\begin{proof}Due to the commutative diagram~\eqref{eq:Mdiagram}, it suffices to
show the assertion about $D_B\im$.
Let $\s_1\ge\ldots\ge\s_m>0$ be the singular values of $B$ and 
$B=UDV^*$ be the singular value decomposition of $B$. 
So $U\in\ug(n+1)$, $V\in\ug(m)$, and $D$ is obtained from 
the diagonal matrix $\Delta:=\diag(\s_1,\ldots,\s_m)$
by appending $s+1$ zero rows. 
Consider the commutative diagram
\begin{equation*}
\begin{array}{rcl}
 \cM   & \stackrel{\im}{\to }& \Grass' \\
   \downarrow &                         &  \downarrow\\
\cM & \stackrel{\im}{\to} &   \Grass 
\end{array} ,
\end{equation*}
where the left vertical arrow is given by 
$B'\mapsto U^*B'V$ and the right vertical arrow is induced by~$U^*$. 
Both vertical arrows are isometries. 
Therefore, for showing the assertion about $D_B\im$, 
we can assume that $B=D$ without loss of generality. 

The matrix $\tilde{I} := I_{n+1,m} = D\Delta^{-1}$ is 
obtained from the unit matrix $I_{m}$ by appending $s+1$ zero rows. 
We can write $D_B\im\colon\cM \to T_L\Grass'$ 
as the composition of $D_{\tilde{I}}\im\colon\cM\to T_L\Grass'$ 
with $\cM\to\cM,\, \dot{B}\mapsto \dot{B}\Delta^{-1}$. 
Note that for all $\dot{B}\in\cM$
\begin{equation}\label{eq:nono}
  \s_1^{-1}\| \dot{B}\|_F \ \le\  \|\dot{B}\Delta^{-1}\|_F \ \le\ \s_m^{-1}\| \dot{B}\|_F .
\end{equation}
According to Section~\ref{se:Grass-conv},  
the derivative $D_{\tilde{I}}\im$ can be interpreted as the orthogonal projection 
$\begin{bmatrix} S \\ R\end{bmatrix}\mapsto R$, 
where $S\in\C^{m\times m}$ and $R\in\C^{(s+1)\times m}$. 
Hence, $\|D_{\tilde{I}}\im(\dot{B})\| \le \|\dot{B}\|_F$ for all
$\dot{B}\in\cM$. 
Moreover, equality holds if $\dot{B}$ lies in the subspace $\cN$ 
defined by $S=0$. 
Combining these insights with \eqref{eq:nono}, the assertion follows.
\end{proof}

The first part of Lemma~\ref{le:diego} implies that 
$\|D_{B}\im\| = \|D_{B^T}\ker\| = \|B^\dagger\|$.
This means that the absolute condition number of the maps $\im$ and $\ker$ 
is given by the matrix condition number $\kappa(B)$; 
compare~\cite[\S14.3]{condition}. 
In the corresponding situation of real matrices,  
$\|D_A\ker\| = \|A^\dagger\|$ was first shown by Armentano~\cite{Arm:10}; 
see also \cite[\S14.3.2]{condition}.

\begin{proof}[Proof of Theorem~\ref{pro:kckA}]
Suppose we are in the situation of Section~\ref{se:intrinsic}. 
So $L\in\Grass$ intersects $Z$ transversally at~$z$ and 
we have open neighborhoods $U\subseteq\Grass$ of $L$, 
$V\subseteq Z$ of $z$, and a smooth solution map 
$\g\colon U \to V$ such that 
$\tilde{L}\cap V = \{\g(\tilde{L})\}$ 
for all $\tilde{L}\in U$. By Definition~\ref{def:condG}, we have 
$\kappa_Z(L,z) = \|D_L\g\|$. 

Composing the maps $\ker$ and~$\g$, we obtain the solution map
$G\colon\ker^{-1}(U) \to V,\, A\mapsto (\g\circ \ker)(A)$. 
Suppose that $L=\ker(A)$. 
By Definition~\ref{def:condA}, 
the submultiplicativity of the operator norm, 
and Lemma~\ref{le:diego}, we obtain 
$$
  \kercond_Z(A,z)  = \|A\| \cdot\| D_AG\| \ \le\ 
 \|A\| \cdot \| D_L\g\| \cdot \| D_A \ker\| 
  = \kappa(A)\cdot \kappa_Z(L,z) .
$$ 
This shows the upper bound stated in Theorem~\ref{pro:kckA}. 
The lower bound follows by combining Lemma~\ref{le:diego} 
with Lemma~\ref{le:subsp-contr} below. 
\end{proof}

\begin{lemma}\label{le:subsp-contr}
Let $\varphi\colon U \to V$ and $\psi\colon V\to W$ be linear 
maps between finite dimensional hermitian vector spaces. 
Suppose that $U'\subseteq U$ is a linear subspace such that 
the  restriction $U'\to V$ of~$\varphi$ is surjective and there is 
$\lambda>0$ such that for all $u\in U'$ we have
$\lambda \|u\| \le \|\varphi(u)\|$.
Then $\lambda \|\psi\| \le \|\psi \circ\varphi\|$. 
\end{lemma}

\begin{proof}
Let $v\ne 0$ be such that $\|\psi(v)\| = \|\psi\| \|v\|$. 
By assumption, there exists $u\in U'$ such that $\varphi(u) =v$. Then 
$\lambda \|u\| \le \|\varphi(u)\| =\|v\|$. Hence 
$\lambda \|\psi\| \|u\| \le \|\psi\| \|v\| = \|\psi(v)\| = \|\psi(\varphi(u))\|$. 
\end{proof}
\medskip

We may also represent $L\in\Grass$ as the image of a full rank matrix $B\in\C^{(n+1)\ti (s+1)}$.
When doing so, we can analogously define an {\em image intersection condition number} 
\begin{equation}\label{eq:def_imcond}
 \imcond_Z(B,z) := \|B\|  \cdot \|D_B \Gamma\| ,
\end{equation}
where $\Gamma$ denotes the locally defined solution map sending matrices $\tilde{B}$ 
close to $B$ to the unique intersection point of $Z \cap\proj(\im\tilde{B})$ close to $z$. 
The same arguments as above then show that 
\begin{equation}\label{eq:imcond}
 \kappa_Z(L,z) \le \imcond_Z(B,z) \le \kappa(B)\cdot\kappa_Z(L,z) .
\end{equation}

\subsection{Angular characterization of intersection condition numbers}
\label{se:ang_char}

We provide the proof of Theorem~\ref{th:kappa-alpha} by proceeding in several steps. 
Let $Z\subseteq\proj^n$ be an irreducible projective variety of dimension~$m$ 
and $z$ be a smooth point of~$Z$. Hence there are homogeneous polynomials
$f_1,\ldots,f_s$ vanishing on $Z$ such that the Jacobian of $f_1,\ldots,f_s$ at $z$ 
has the rank $s=n-m$. Then, in an open neighborhood of $z$, the variety~$Z$ 
equals the zero set of $f_1,\ldots,f_s$; cf.~\cite[Cor.~(1.20)]{mumford}. 
It will be convenient to use the coordinate functions 
$T,X_1,\ldots,X_s,Y_1,\ldots,Y_m$.
By unitary invariance, we may assume without loss of generality that 
$L$ is the zero set of $Y_1,\ldots,Y_m$ and 
$z=(1 \colon 0\colon\ldots\colon 0)$. 
We represent $z$ by the affine point $\z=(1, 0,\ldots,0)$.  

The tangent space $T_z Z$ consists of the vectors 
$(0,\dot{x},\dot{y})\in\C^{n+1}$ such that 
$$
 \partial_X f (\z)\, \dot{x} + \partial_Y f (\z)\, \dot{y} = 0 ,
$$ 
since $\partial_{T} f(\z)=0$ by Euler's relation. 
We may assume without loss of generality 
that $Z$ and $L$ meet transversally at~$z$. 
This means that the matrix $\partial_X f(\z) \in\R^{s\ti s}$ is invertible.
Setting $N:=-(\partial_X f(\z))^{-1}\partial_Y f(\z)$ 
we conclude that 
\begin{equation}\label{eq:tang-space-Z}
 T_{z}Z =\{ (0,N\dot{y},\dot{y}) \mid \dot{y} \in\C^m\} . 
\end{equation}

\begin{lemma}\label{le:angle-tangspace}
The minimum angle $\a$ between $T_{z}Z$ and $T_{z}L$
satisfies $\cot(\a) = \|N\|$. 
In particular, $\sin(\a)= (1+\|N\|^2)^{-\frac12}$. 
\end{lemma}

\begin{proof}
By~\eqref{eq:tang-space-Z} we have 
$T_{z} Z \simeq\{ (N\dot{y},\dot{y}) \mid \dot{y} \in\C^m\}$
(omitting the first component, which is always zero). 
Moreover, $T_{z} L \simeq \C^s\ti \{0\}^m$. 

Suppose first $s\le m$ and consider the singular value decomposition $N=UDV^*$, 
where $U\in\ug(s),V\in\ug(m)$, and $D$ is the diagonal matrix 
with entries $\s_1\ge\cdots\ge\s_s$, with $m-s$ zero columns appended. 
The $\s_i$ are the singular values of $N$ and $\|N\|=\s_1$. 
The unitary map $(x,y)\mapsto (U^*x,V^*y)$ sends $T_{z}Z$ to 
$W_1:=\{ (D\dot{y},\dot{y}) \mid \dot{y} \in\C^m\}$ and leaves  
$W_2:=\C^s\ti \{0\}^m$ invariant. 
Hence the minimum angle $\a$ between $T_{z}Z$ and $T_{z}L$
equals the minimum angle between $W_1$ and $W_2$. 
We next compute this angle. 

Let $\Delta\in\R^{m\ti m}$ denote the diagonal matrix with the 
entries $(1+\s_1^2)^{-\frac12},\ldots,(1+\s_s^2)^{-\frac12},1,\ldots,1$. 
Then the columns of 
$\begin{bmatrix}
D\Delta\\
\Delta
\end{bmatrix}$
and 
$\begin{bmatrix}
I_s\\
0
\end{bmatrix}$ 
form an orthonormal basis of $W_1$ and $W_2$, respectively. 
By the description of the principal angles in Section~\ref{se:geo-char}, 
the minimum angle between $W_1$ and $W_2$ is obtained as the arccosine 
of the spectral norm of 
$\begin{bmatrix}
I_s & 0 
\end{bmatrix}
\cdot 
\begin{bmatrix}
D\Delta\\
\Delta
\end{bmatrix}
=D\Delta$.
It is easily checked that $\frac{\s_1}{\sqrt{1+\s_1^2}}$ 
is the largest among the $\frac{\s_i}{\sqrt{1+\s_i^2}}$. 
Therefore, $\cos\a= \frac{\s_1}{\sqrt{1+\s_1^2}}$, 
which implies the assertion $\cot\a=\s_1$. 

The case $s>m$ is treated analogously. 
The second assertion follows using $\sqrt{1+\cot(\a)^2}=(\sin(\a))^{-1}$.
\end{proof}

We note that the minimum angle $\a$ is positive iff $Z$ and $L$ 
meet transversally at $z$. We can therefore assume $\a>0$. 
According to Lemma~\ref{le:angle-tangspace}, 
we need to show that $\kappa_Z(L,z) = \sqrt{1+\|N\|^2}$
in order to complete the proof of Theorem~\ref{th:kappa-alpha}.

\begin{proof}[Proof of Theorem~\ref{th:kappa-alpha}]
We go back to the setting of Section~\ref{se:intrinsic},
put $\Grass := \Grass(s,\proj^n)$, $\ug := \ug(n+1)$,  
and consider the solution map $\g\colon \Grass\supseteq U\to V\subseteq\proj^n$ 
at $(L,z)$, cf.~\eqref{eq:gamma}, where
$L=\C^{s+1} \times\{0\}^m$ and $z=(1:0:\ldots : 0)$. 
Composing $\g$ with the map $\pi\colon\ug\to\Grass$,
defined analogously as in~\eqref{eq:pi-map}, 
we obtain the lifted solution map 
$\tilde{\g}\colon \tilde{U} \to Z$ defined on the open neighborhood 
$\tilde{U}:=\pi^{-1}(U)$ of $I :=I_{n+1}$.  
Note that $\tilde{\g}(I) = z$. 
By our definition of the Riemannian metric on $\Grass$, the derivative 
$D_I\pi$ of $\pi$ maps the unit ball in $T_I\ug$ onto the unit ball in $T_{L}\Grass$.
Hence, Definition~\ref{def:condG} implies that 
\begin{equation}\label{eq:char-cond-lifted}
 \kappa_Z(L,z) =  \| D_L \g\| = \|D_I \tilde{\g} \| . 
\end{equation}

We suppose now that $t\mapsto Q(t)$ is a smooth curve in $\ug$ such that $Q(0)=I$ 
and denote by $\dot{Q}$ its derivative at $t=0$. 
We write $Q(t)$ as a $3\ti 3$ block matrix according to the decomposition $n+1=1+s+m$, 
$$
 Q(t)=\begin{bmatrix} Q_{00}(t) & Q_{01}(t) & Q_{02}(t) \\
      Q_{10}(t) & Q_{11}(t) & Q_{12}(t) \\ Q_{20}(t) & Q_{21}(t) & Q_{22}(t)\end{bmatrix} ,
$$
where 
$Q_{00}(t)\in \C$, $Q_{01}(t)\in\C^{1\ti s}$, $Q_{02}(t)\in\C^{1\ti m}$,
$Q_{10}(t) \in \C^{s\ti 1}$, $Q_{11}(t) \in \C^{s\ti s}$, $Q_{12}(t) \in \C^{s\ti m}$,
$Q_{20}(t) \in \C^{m\ti 1}$, $Q_{21}(t) \in \C^{m\ti s}$, $Q_{22}(t) \in \C^{m\ti m}$.
Moreover, we write $Q_{ij} := Q_{ij}(0)$ and denote by $\dot{Q}_{ij}$ the 
derivative of $\dot{Q}_{ij}(t)$ at $t=0$. 
Note that $\dot{Q}_{00} \in \R i$. 
This defines the curve $t\mapsto L(t):=\pi(Q(t))$ in $\Grass$, 
where $L(t)$ is the span of the first $s+1$ columns of $Q(t)$. 

Let $\Sph := S(\C^{n+1})$ denote the unit sphere of $\C^{n+1}$ 
and consider the canonical map 
$p\colon \Sph\to \proj^n$.
By lifting the intersection point 
$z(t) = \tilde{\g}(Q(t)) \in Z \cap L(t)$, 
we get a smooth function $t\mapsto \z(t) \in \Sph$
with 
$z(t) = p(\z(t))$ and $\z(0) = \z= (1,0,\ldots,0)$. 
(Note that representatives in $\Sph$ of points in $\proj^n$ are only 
determined up to a complex scalar of absolute value one.) 
We can write 
\begin{equation}\label{eq:def-z}
 \z(t) = 
\begin{bmatrix} 
\z_0(t)\\
\z_1(t)\\
\z_2(t)
\end{bmatrix}
= 
\begin{bmatrix} 
Q_{00}(t) & Q_{01}(t) \\
Q_{10}(t) & Q_{11}(t) \\
Q_{20}(t) & Q_{21}(t) 
\end{bmatrix} 
\begin{bmatrix} 
u(t) \\
v(t)
\end{bmatrix} 
\end{equation}
for smooth functions $u(t)\in\C$ and $v(t)\in\C^{s}$ of~$t$.
Moreover, $\z_0(0)=1$, $\z_1(0)=0$, $\z_2(0)=0$ and 
$u(0) = 1$, $v(0)=0$.
Let us write $\dot{\z}_i$ for the derivative of $\z_i(t)$ at $t=0$. 
Taking derivatives in~\eqref{eq:def-z}, we get 
\begin{equation}
\begin{bmatrix} 
\dot{\z}_0\\
\dot{\z}_1\\
\dot{\z}_2
\end{bmatrix}
= 
\begin{bmatrix} 
\dot{Q}_{00} & \dot{Q}_{01} \\
\dot{Q}_{10}& \dot{Q}_{11}\\
\dot{Q}_{20} & \dot{Q}_{21}
\end{bmatrix} 
\begin{bmatrix} 
1 \\
0
\end{bmatrix} 
+ 
\begin{bmatrix} 
1 & 0 \\
0 & I\\
0 & 0
\end{bmatrix} 
\begin{bmatrix} 
\dot{u} \\
\dot{v}
\end{bmatrix} 
=
\begin{bmatrix} 
\dot{Q}_{00} + \dot{u} \\
\dot{Q}_{10} + \dot{v} \\
\dot{Q}_{20}
\end{bmatrix} 
\end{equation}
and we obtain
$\dot{\z}_2 = \dot{Q}_{20}$. 

The fact $z(t) \in L(t)\cap Z$ can be expressed as $f(\z(t))=0$.
Taking the derivative at $t=0$
(and recalling $\partial_T f(\z) = 0$) 
implies 
$\partial_X f(\z)\, \dot{\z}_1 + \partial_Y f(\z)\, \dot{\z}_2 = 0$, 
hence (recall $N:=-(\partial_X f(\z))^{-1}\partial_Y f(\z)$) 
$$
\dot{\z}_1 = N \dot{\z}_2 = N\dot{Q}_{20} .
$$

The derivative of the canonical map 
$p\colon\Sph\to\proj^n$ at $\z$ is  
the orthogonal projection 
$T_\z\Sph\to T_z\proj^n$ 
along $\R i\z$; see \cite[Lemma~14.9]{condition}.
Therefore, the derivative $\dot{z}$ 
is obtained from 
$\dot{\z}$ 
by applying the orthogonal projection along $\R iz$.
Note that $\dot{\z}_0 \in i\R$ and $\dot{z}_0 =0$. 
We obtain 
\begin{equation}
\dot{z} = 
\begin{bmatrix}
0  \\
\dot{z}_1\\
\dot{z}_2 
\end{bmatrix}
= 
\begin{bmatrix}
0  \\
N \dot{Q}_{20}\\
\dot{Q}_{20} 
\end{bmatrix} .
\end{equation}
Equation~\eqref{eq:char-cond-lifted} tells us that 
the condition number $\kappa_Z(L,z)$ equals 
the maximum of $\|\dot{z}\|$, taken over all 
$\dot{Q}\in T_I\ug$ of norm at most one. 
This norm condition amounts to
$\|\dot{Q}_{20}\| \le 1$; cf.~\eqref{eq:ug-can-metric}.
Therefore, $\kappa_Z(L,z) = \sqrt{\|N\|^2 +1}$, which 
completes the proof of Theorem~\ref{th:kappa-alpha}. 
\end{proof} 
 
\subsection{Proof of the Condition Number Theorem} 
\label{se:pf_CNT}

We first provide the proof of an auxiliary result that was stated in the introduction. 

\begin{proof}[Proof of Lemma~\ref{le:Sigma}]
Let $\Reg(Z)$ denote the set of smooth points of $Z$.
We denote by $\Sigma'$ the Zariski closure of 
$\Sigma'_o :=\{(z,L) \in \Reg(Z)\ti\Grass\mid L\in\Sigma_z(Z)\}$
in $Z\times\Grass$. 
The image of $\Sigma'$ under the projection 
$\pi_2\colon Z\ti\Grass \to \Grass,\, (\z,\tilde{L})\mapsto \tilde{L}$
is closed and, by Definition~\ref{def:hurwitz}, 
equals the Hurwitz variety $\Sigma(Z)$. 

Let now $L\in\Sigma(Z)\setminus\Irrel(Z)$. Then there exists $z\in Z$ such that 
$(z,L) \in\Sigma'$. Since $L\not\in\Irrel(Z)$, we have $z\in\Reg(Z)$.
The usual vanishing of minors condition implies that 
$\Sigma'_o$ is a Zariski closed subset of $\Reg(Z)\ti\Grass$. 
With $z\in\Reg(Z)$ this implies that $(z,L)\in\Sigma'_o$ 
and we see that indeed $L\in\Sigma_z(Z)$. 
\end{proof}

We continue with a general reasoning. 
Let $E$ be a finite dimensional hermitian vector space and 
consider the Grassmann manifold $\Grass(k,E)$ of $k$-dimensional 
linear subspaces of $E$. For a fixed linear subspace $T\subseteq E$ 
we consider the {\em Schubert variety}
\begin{equation}\label{eq:def-Sigma-T}
S_T := \big\{ W\in\Grass(k,E) \mid W\cap T\ne 0 \big\} .
\end{equation}
$S_T$ is a closed subset of $\Grass(k,E)$ since 
$W\cap T\ne 0$ is equivalent to 
$\dim (W+T) < \dim W + \dim T$,
which can be expressed by the vanishing of minors.

We write $d_p(W,S_T) := \min \{d_p(W,W') \mid W' \in S_T\}$ 
for the minimal projection distance between $W$ and $S_T$.
Similarly, we define the minimal geodesic distance $d_g(W,S_T)$.

\begin{proposition}\label{pro:char-min-angle}
For any $W\in\Grass(k,E)$ we have 
$$
  d_p(W,S_T) = \sin d_g(W,S_T) = \sin\ang_{\min}(W,T) .
$$
\end{proposition}

\begin{proof}
It suffices to prove the assertion in the case where $T$ is one-dimensional.
In the analogous situation of a euclidean vector space $E$, this was shown in
\cite[Lemma 3.2]{amel-pbuerg:11a}. 
It is straightforward to check that the proof given there extends to the situation of a hermitian vector space $E$.
\end{proof}

\begin{proof}[Proof of Theorem~\ref{th:CNT}]
We return to the setting of Theorem~\ref{th:CNT} and will apply 
Proposition~\ref{pro:char-min-angle} to the hermitian vector space $E:=T_z\proj^n$ and 
its subspace $T:=T_zZ$. 
Recall that $\Grass_z$ denotes the set of $L\in\Grass(s,\proj^n)$ passing through $z$. 
We have the bijection  
\begin{equation}\label{eq:G-bij}
  \Grass_z \longrightarrow \Grass(s,T_z\proj^n),\, L \mapsto T_z L := \widehat{L}/\C z
\end{equation}
If $L_1,L_2 \in \Grass_z$, then the list of principal angles between $T_z L_1$ and $T_z L_2$ is obtained 
from the list of principal angles between $L_1$ and $L_2$ by removing a zero; 
cf.~Lemma~\ref{le:pr-ang-red}. 
It follows that the bijection~\eqref{eq:G-bij} preserves the projective distance $d_p$ as 
well as the geodesic distance $d_g$. 

By its definition, the local Schubert variety $\Sigma_z(Z)$ of $Z$ at $z$
is mapped to $S_T$ under the bijection \eqref{eq:G-bij}. 
Since this is an isometry, 
we obtain with Proposition~\ref{pro:char-min-angle} that 
$d_p(L,\Sigma_z) = d_p(T_z L, S_T) = \sin\a$, 
where $\a:= \ang_{\min}(T_z L,T_z Z)$.
Similarly, $d_g(L,\Sigma_z) = d_g (T_z L, S_T) = \a$.  
Finally, Theorem~\ref{th:kappa-alpha} tells us that 
$1/\kappa_Z(L,z) = \sin \a$.
This completes the proof of Theorem~\ref{th:CNT}. 
\end{proof}
 
\section{Expressing volume in terms of degree}
\label{se:vol-HS}

We provide here the proof of Theorem~\ref{th:HlinC}. 
Recall from~\eqref{eq:plucker} the Pl\"ucker embeddding
$\iota\colon\Grass \hookrightarrow \proj(\Lambda^{s+1}\C^{n+1})$.
We have defined a Riemannian metric on $\Grass$ in Section~\ref{se:intrinsic}.  
Moreover, on the projective space $\proj(\Lambda^{s+1}\C^{n+1})$, 
we have the Riemannian metric resulting from the Fubini-Study metric, cf.~\cite{shaf-2}. 
We show now that these metrics are compatible.

\begin{lemma}\label{le:isometric}
The Pl\"ucker embeddding $\iota$ is isometric. 
\end{lemma}

\begin{proof} 
The composition of $\iota$ with the map~$\pi\colon\ug(n+1)\to\Grass$ from~\eqref{eq:pi-map}
is described by the map 
$\varphi\colon\ug(n+1) \to \Lambda^{s+1}\C^{n+1}, Q\mapsto y$, 
that maps a matrix $Q$ to the list $(y_I)$ of the $(s+1)$-minors of the submatrix of $Q$ 
consisting of its first $s+1$ columns. 
The unit matrix $I$ is mapped to the vector $e_0=(1,0,\ldots,0) \in \Lambda^{s+1}\C^{n+1}$. 
By unitary invariance, it is sufficient to prove that $\varphi$ 
is a Riemannian submersion at $I$. The derivative 
$D_I\varphi\colon T_I\ug(n+1) \to \Lambda^{s+1}\C^{n+1}$,
restricted to the space of 
skew-hermitian matrices of the form~\eqref{eq:TS-Gr}, is given by 
$$
\dot{Q} =\begin{bmatrix}
0 & -R^* \\
R & 0
\end{bmatrix}
\mapsto (\dot{y}_I), \quad 
 \dot{y}_I = \left\{
    \begin{array}{ll} r_{ij} & \mbox{ if $I= \{1,\ldots,s+1\} \setminus \{j\} \cup \{s+i\}$}\\
                           0      &  \mbox{ otherwise.}
    \end{array} \right.  ,           
$$
where $R=(r_{ij}) \in \C^{m\times (s+1)}$. 
This map is isometric according to our definition~\eqref{eq:def-hip} 
of the inner product on $T_I \Grass$ and the definition of the 
Fubini-Study metric on $T_{e_0}\proj^n$.
\end{proof}

We recall a well known and fundamental result in complex algebraic geometry; 
see \cite[\S5.C]{mumford} or \cite[Chap.~VIII, \S4.4]{shaf-2}. 

\begin{theorem}\label{th:degvol}
An irreducible complex projective variety~$Z$ of dimension~$n$ satisfies
$\vol(Z) = \deg Z \cdot \vol(\proj^n)$. \hfill\qed
\end{theorem}

\begin{proof}[Proof of Theorem~\ref{th:HlinC}]
Let $\Hy$ be an algebraic hypersurface of $\Grass$ 
and put $N := \dim\Grass$. 
Lemma~\ref{le:isometric} gives $\vol(\Hy) = \vol(\iota(\Hy))$. 
Applying Theorem~\ref{th:degvol} to $Z=\iota(\Hy)$, we get 
$\vol(\iota(\Hy)) = \deg\iota(\Hy) \cdot\vol(\proj^{N-1})$.  
Using~\eqref{eq:rdeg}, we obtain 
$$
 \vol(\Hy) = \rdeg\Hy \cdot \deg\iota(\Grass)\cdot\vol(\proj^{N-1}) .
$$  
In particular, $\vol(\Hyl) = \deg\iota(\Grass) \cdot\vol(\proj^{N-1})$
and the first assertion follows.

For the second assertion, we apply 
Lemma~\ref{le:isometric} and Theorem~\ref{th:degvol} to 
$Z=\iota(\Grass)$ and get
$\vol(\Grass) = \vol(\iota(\Grass)) = \deg\iota(\Grass) \cdot\vol(\proj^N)$. 
This implies 
$\vol(\Hyl)/\vol(\Grass) = \vol(\proj^{N-1})/\vol(\proj^N) = N/\pi$.
\end{proof}


\end{document}